\documentclass[11pt,reqno]{amsart}
\usepackage[utf8]{inputenc}  
\usepackage[T1]{fontenc}
\usepackage{graphicx}
\usepackage{cite}
\usepackage{amssymb, amsmath}
\usepackage[normalem]{ulem}

\usepackage{fixltx2e}
\usepackage{newunicodechar}
\usepackage{color,xcolor}
\usepackage[normalem]{ulem}

\usepackage{hyperref}

\newtheorem{thm}{Theorem}[section]
\newtheorem{lem}[thm]{Lemma}
\newtheorem{cor}[thm]{Corollary}
\newtheorem{prop}[thm]{Proposition}
\newtheorem{rem}[thm]{Remark}

\newtheorem{rems}[thm]{Remarks}

\theoremstyle{definition}

\theoremstyle{remark}

\numberwithin{equation}{section}

\DeclareMathAlphabet{\mathpzc}{OT1}{pzc}{m}{it}

\newcommand{\N}{\mathbb N}
\newcommand{\R}{\mathbb R}

\newcommand{\ve}{\varepsilon}

\newcommand{\rd}{\mathrm{d}}

\begin{document}

 \author[L.S. Schmitz]{Lina Sophie Schmitz}
\address{Leibniz Universit\"at Hannover,
Institut f\"ur Angewandte Mathematik,
Welfengarten 1,
30167 Hannover,
Germany.}
\email{schmitz@ifam.uni-hannover.de}

\author[Ch. Walker]{Christoph Walker}
\address{Leibniz Universit\"at Hannover,
Institut f\"ur Angewandte Mathematik,
Welfengarten 1,
30167 Hannover,
Germany.}
\email{walker@ifam.uni-hannover.de}

\title[]{Recovering Initial States  in Semilinear Parabolic Problems from Time-Averages}

\begin{abstract}
Well-posedness of certain semilinear parabolic problems with nonlocal initial conditions is shown in time-weighted spaces. The result is applied to recover the initial states in semilinear parabolic problems with nonlinearities of superlinear behavior near zero from small time-averages  over arbitrary time periods.
\end{abstract}


\maketitle
%%%%%%%%%%%%%%%%%%%%%%%%%%%%%%%%%%%%%%%%%
  %%%%%%%%%%%%%%%%%%%%%%%%%%%%%%%%%%%%%%%%%
  %%%%%%%%%%%%%%%%%%%%%%%%%%%%%%%%%%%%%%%%%
%%%%%%%%%%%%%%%%%%%%%%%%%%%%%%%%%%%%
 %%%%%%%%%%%%%%%%%%%%%%%%%%%%%%%%%%%
 %%%%%%%%%%%%%%%%%%%%%%%%%%%%%%%%%%
\section{Introduction}\label{Sec:1}
%%%%%%%%%%%%%%%%%%%%%%%%%%%%%%%%%%%%%%%%%
  %%%%%%%%%%%%%%%%%%%%%%%%%%%%%%%%%%%%%%%%%
  %%%%%%%%%%%%%%%%%%%%%%%%%%%%%%%%%%%%%%%%%
%%%%%%%%%%%%%%%%%%%%%%%%%%%%%%%%%%%%
 %%%%%%%%%%%%%%%%%%%%%%%%%%%%%%%%%%%
 %%%%%%%%%%%%%%%%%%%%%%%%%%%%%%%%%%

The problem of recovering the unknown initial state $u(0,x)$ in linear uniformly parabolic diffusion equations
\begin{subequations}
\label{DD}
\begin{align}
\partial_t u-\mathrm{div}\big(d(x)\nabla u\big)+c(x) u &= f(t,x)\,,\qquad (t,x)\in (0,T]\times\Omega\,,\\
 u&=0\,,\qquad (t,x)\in (0,T]\times\partial\Omega\,,
\end{align}
from time-averages
\begin{equation}\label{timeav}
a\, u(T,x)+\int_0^T  b(t)\, u(t,x)\,\rd t=M(x)\,,\quad x\in\Omega\,,
\end{equation}
\end{subequations}
is addressed in \cite{Dokuchaev} in the Hilbert space setting of $L_2(\Omega)$. Here, $T \in (0,\infty)$ is a fixed given time. Paraphrasing the main result of \cite[Theorem~1]{Dokuchaev} it is shown therein that this problem admits a unique weak solution $u$ for each given $M\in H^2(\Omega)\cap \mathring{H}^1(\Omega)$ provided that $a\ge0$, $b\ge 0$ with $\mathrm{essinf}_{t\in [0,t_0]} b(t)>0$ for some $t_0\in (0,T)$, and that~$f$ is a (suitable) $L_2$-function. 
 Well-posedness in an $L_p$-setting of a linear fractional diffusion equation with Riemann–Liouville derivative including a time-average condition like~\eqref{timeav} is established in~\cite{TachEtal_MMAS23}. 

We shall consider herein strong solutions to general abstract semilinear problems with nonlinear functions $f=f(u,t)$ not restricting to a Hilbert space framework. More precisely, we focus our attention on semilinear parabolic problems
\begin{equation}\label{E}
u'=Au+f(u,t)\,,\quad t\in(0,T]\,,\qquad a\, u(T)+\int_0^T  b(t)\, u(t)\,\rd t=M\,,
\end{equation}
involving a generator $A$ of an analytic semigroup on a Banach space $E_0$, a nonlinear function \mbox{$f=f(u,t)$}, and a nontrivial weight function $b$. Roughly speaking, we shall show that for  Lipschitz continuous functions $f$ (depending locally or nonlocally on $u$) with superlinear behavior near zero one can recover the (unique) initial state from small values of time-averages over arbitrary time periods. That is, problem~\eqref{E} is well-posed for small values of $M$. The smallness condition is due to the fact that $T>0$ is {\it a priori} fixed and one thus seeks for global solutions to a nonlinear problem. Since $b=0$ yields an ill-posed final value problem  (in the sense that solutions exist only for particular smooth data), we restrict to the case $b\not=0$.

A conceptually different situation than~\eqref{E} occurs if initial and final values are related 
\begin{equation}\label{E100}
u'=Au+f(u,t)\,,\quad t\in(0,T]\,,\qquad u(0)-bu(T)=M\,,
\end{equation}
or if the initial value is related to nonlocal time-averages
\begin{equation}\label{E200}
u'=Au+f(u,t)\,,\quad t\in(0,T]\,,\qquad u(0)+\int_0^T b(t)\, u(t)\,\rd t=M\,.
\end{equation}
 Results regarding solvability, uniqueness, and regularity of solutions to the corresponding linear cases of \eqref{E100}-\eqref{E200} with $f$ independent of $u$ are derived in Hilbert space settings in~\cite{Dokuchaev4,Dokuchaev6,Dokuchaev7}
 and for nonlinear problems in~\cite{TrietEtal_MMAS21} (see also the references therein). In~\cite{Dokuchaev4,Dokuchaev6,Dokuchaev8} also the connection to stochastic differential equations is exploited. Moreover, for recent results regarding final value problems for semilinear (and linear) equations with fractional diffusion and Riemann–Liouville derivative or conformable time derivative we refer e.g. to~\cite{TranEtal_MMAS20,AuEtal_JIEA_20,TrietEtal_MMAS20,TuanEtAl_MMAS20} and the references therein.\\

Despite the conceptual differences, we shall treat problems \eqref{E}, \eqref{E100}, and~\eqref{E200} simultaneously; that is, as special cases of the following semilinear problem with nonlocal initial condition
\begin{equation}\label{EE}
u'=Au+f(u,t)\,,\quad t\in(0,T]\,,\qquad u(0)=\Sigma_T(u)\,,
\end{equation}
where $\Sigma_T(u)$ is a (nonlocal) function of $u$ adapted to the specific problem. Under suitable conditions on the nonlinearities $f$ and $\Sigma_T$ we shall prove that~\eqref{EE} is globally well-posed and implies the well-posedness of each of the problems~\eqref{E}, \eqref{E100}, and~\eqref{E200}.
 As we  shall later see, in the context of problem \eqref{E} the initial condition $\Sigma_T(u)$ belongs merely to the ambient space~$E_0$ and a corresponding solution can thus only be continuous at $t=0$ in the norm of $E_0$. In order to put only mild conditions on the nonlinearity $f$ when investigating problem~\eqref{EE} with  $\Sigma_T(u)\in E_0$, this necessitates the use of a time-weighted space as a phase space for solutions. Such spaces have been used in various research papers in different contexts, e.g. see \cite{PSW18,MW24,Yagi10} and the references therein.

%%%%%%%%%%%%%%%%%%%%%%%%%%
%%%%%%%%%%%%%%%%%%%%%%%%%%
\section{Main Results}
%%%%%%%%%%%%%%%%%%%%%%%%%%
%%%%%%%%%%%%%%%%%%%%%%%%%%

%%%%%%%%%%%%%%%%%%%%%%%%%%
\subsection*{Notations and Assumptions}
%%%%%%%%%%%%%%%%%%%%%%%%%%

Throughout this paper we assume that $E_0$ and $E_1$ are Banach spaces such that $E_1$ is dense and continuously embedded in $E_0$. Moreover, we fix 
\begin{equation}\label{Gen}
A\in\mathcal{H}(E_1,E_0)\,,
\end{equation} 
that is, $A$ is the generator of an analytic semigroup $(e^{tA})_{t\ge 0}$ on $E_0$ with domain $E_1$.\\

Given $\theta\in (0,1)$, let $(\cdot,\cdot)_\theta$ be an arbitrary admissible interpolation functor of 
 exponent $\theta$  (see e.g. \cite[I.Section~2.11]{LQPP}) and set $E_\theta:= (E_0,E_1)_\theta$ for the corresponding interpolation space with norm~$\|\cdot\|_\theta$.\\

 Given $\mu\in \R$ we denote by $C_{\mu}((0,T],E_\theta)$ the Banach space of all functions
$u\in C((0,T],E_\theta)$ such that 
$t^{\mu} u(t)\rightarrow 0$ in $E_\theta$ as $t\rightarrow 0^+$, equipped with the norm
\begin{equation*}
u\mapsto \|u\|_{C_{\mu}((0,T],E_\theta)} := \sup\left\{ t^{\mu}\, \|u(t)\|_\theta \,:\, t\in (0,T]\right\}\,.
\end{equation*}
% Note that 
%\begin{equation}\label{Emb}
%C_{\mu}((0,T],E)\hookrightarrow C_{\nu}((0,T],E)\,,\quad \mu\le \nu\,.
%\end{equation}

In the following, we assume that there are
\begin{subequations}\label{A}
\begin{equation}\label{A1}
\gamma\in [0,1]\,, \qquad \theta\in (0,1]\,,\qquad (\gamma,\theta)\not=(0,1)\,,\qquad \nu \in [0,1)\,,\qquad \ell>0\,,
\end{equation}
for which
$$
f:C_{\theta}\big((0,T],E_\theta\big)\to C_{\nu}\big((0,T],E_\gamma\big)
\qquad \text{and}\qquad \Sigma_T: C_{\theta}\big((0,T],E_\theta\big)\to E_0
$$ 
are Lipschitz mappings satisfying 
\begin{equation}\label{A2}
\|f(v)-f(w)\|_{C_{\nu}((0,T],E_\gamma)}\le c_T\,L^{\ell}\, \|v-w\|_{C_{\theta}((0,T],E_\theta)}\,,\qquad f(0)=0\,,
\end{equation}
\end{subequations}
and
\begin{equation}\label{A3}
\|\Sigma_T(v)-\Sigma_T(w)\|_0\le c_T\,L^{\ell}\, \|v-w\|_{C_{\theta}((0,T],E_\theta)}
\end{equation}
for $L \in (0,1)$ and $v,w\in C_{\theta}((0,T],E_\theta)$ with 
$$
\|v\|_{C_{\theta}((0,T],E_\theta)}+\|w\|_{C_{\theta}((0,T],E_\theta)}\le L\,,\qquad \sup_{t\in (0,T]}\big(\|v(t)\|_0+\|w(t)\|_0\big)\le L\,,
$$
and where $c_T>0$ is a suitable constant.
We use the notation
$$
f(u,t):=f(u)(t)\,,\qquad t\in (0,T]\,,\quad u\in C_{\theta}\big((0,T],E_\theta\big)\,.
$$
Moreover, we write $\mathcal{L}(E,F)$ for the space of bounded linear operators between two Banach spaces $E$ and $F$, while $\mathcal{L}is(E,F)$ stands for the subspace of isomorphisms. By $\mathcal{K}(E)$ we denote the subspace of $\mathcal{L}(E):=\mathcal{L}(E,E)$ consisting of compact operators.\\

Under these assumptions and notations we can state the main results regarding the general problem \eqref{EE} as well as its specifications~\mbox{\eqref{E}-\eqref{E200}}.

%%%%%%%%%%%%%%%%%%%%%%%%%%
\subsection*{Main Result for the General Problem~\eqref{EE}}
%%%%%%%%%%%%%%%%%%%%%%%%%%

Regarding problem~\eqref{EE} with an a priori given final time $T>0$ we prove the following well-posedness result for small values of $\Sigma_T(0)$:

\begin{thm}\label{T1}
Let $T>0$. Assume \eqref{Gen}, \eqref{A}, and \eqref{A3}. Then there is $m(T)>0$ such that~\eqref{EE} has a unique global mild solution
$$
u\in C\big([0,T],E_0\big)\cap C_\theta\big((0,T], E_\theta\big)
$$
provided that $\|\Sigma_T(0)\|_0\le m(T)$. Moreover, if $\gamma>0$, then 
$$
u\in C^1\big((0,T],E_0\big)\cap C\big((0,T], E_1\big)
$$
is a strong solution to~\eqref{EE}.
\end{thm}

If $T>0$ is chosen small, one can relax assumption \eqref{A2} in the sense that the condition $f(0)=0$ is not required. We state this explicitly:

\begin{rem}\label{R1}
If $f(0)\not=0$ there are $T_0>0$ and $m(T_0)>0$ such that for every $T\in (0,T_0)$ with $\|\Sigma_T(0)\|_0\le m(T_0)$ there is a unique mild solution
$$
u\in C\big([0,T],E_0\big)\cap C_\theta\big((0,T], E_\theta\big)
$$ 
to~\eqref{EE} that is a strong solution if additionally $\gamma>0$. The proof is along the lines of the proof of Theorem~\ref{T1}.
\end{rem}

 It is worth emphasizing that assumption~\eqref{A} on the nonlinearity $f=f(u)$ includes the case of a nonlocal dependence on $u$ with respect to time (see Section~\ref{Sec7} for such examples). Nevertheless, the result also includes the particular case of local time dependence that we state explicitly for completeness. To this end, assume that there are
\begin{subequations}\label{aaa}
\begin{equation}\label{A1Y}
 \ell>0\,,\qquad 0<\theta<\frac{1}{\ell+1}\,,
\end{equation}
such that $f:E_\theta\to E_0$ 
satisfies
\begin{equation}\label{F2}
\|f(v)-f(w)\|_{0}\le c \left(\|v\|_\theta^\ell+\|w\|_\theta^\ell\right) \|v-w\|_{\theta}\,,\qquad v,w\in E_\theta\,,\qquad f(0)=0\,,
\end{equation}
\end{subequations}
for some constant $c>0$.
With the understanding $f(u,t)=f(u(t))$ we then have for this local case: 

\begin{cor}\label{C11}
Let $T>0$. Assume \eqref{Gen},~\eqref{A3}, and \eqref{aaa}. Then there is $m(T)>0$ such that~\eqref{EE} has a unique global strong solution
$$
u\in C^1\big((0,T],E_0\big)\cap C\big((0,T], E_1\big)\cap  C_\theta\big((0,T], E_\theta\big)\cap  C\big([0,T],E_0\big)
$$
provided that $\|\Sigma_T(0)\|_0\le m(T)$.
\end{cor}

Note that growth conditions like~\eqref{F2} are common in the context of time-weighted spaces~\cite{PSW18,MW24,Yagi10}. Also note that we obtain in Corollary~\ref{C11} strong solutions even though $f$ maps only into~$E_0$.

\subsection*{Applications to Initial State Recovering Problems}

In order to apply Theorem~\ref{T1} to
 problem~\eqref{E}, we formulate the latter in the form of~\eqref{EE}. To this end, suppose that $u$ is a mild solution to~\eqref{E} for a given $M\in E_1$; that is, 
\begin{equation}\label{vdk}
u(t)=e^{tA}u(0)+\int_0^t e^{(t-s)A} f(u,s)\,\rd s\,,\qquad t\in [0,T]\,.
\end{equation}
Plugging this into the nonlocal condition
$$ 
au(T)+\int_0^T  b(t)\, u(t)\,\rd t=M
$$
 yields
\begin{align*}
M&= a\,e^{TA}u(0)+a\int_0^T e^{(T-s)A} f(u,s)\,\rd s+\int_0^Tb(t)e^{tA}u(0)\,\rd t+\int_0^Tb(t)\int_0^t e^{(t-s)A} f(u,s)\,\rd s\,\rd t\,.
\end{align*}
Hence, introducing the notion
\begin{equation}\label{Phi}
\Phi_T:=a\,e^{TA}+\int_0^Tb(t)\,e^{tA}\,\rd t \in \mathcal{L}(E_0,E_1)
\end{equation}
and 
\begin{equation}\label{Psi}
\Psi_T g:=a\int_0^T e^{(T-s)A} g(s)\,\rd s+ \int_0^Tb(t)\int_0^t e^{(t-s)A} g(s)\,\rd s\,\rd t\,,
\end{equation}
we (formally) obtain that $u(0)$ is given by
\begin{equation}\label{ulin}
u(0)=\Phi_T^{-1}\big(M-\Psi_T(f(u))\big)\,.
\end{equation}
That is, $u$ is a mild solution to~\eqref{E} if and only if $u$ is a mild solution to~\eqref{EE} with
\begin{equation}\label{u0}
\Sigma_T(u)=\Phi_T^{-1}\big(M-\Psi_T(f(u))\big)\,.
\end{equation}
We will investigate properties of the operators $\Phi_T$ and $\Psi_T$ later on. In fact, in order to guarantee the invertibility of $\Phi_T$ under an assumption that is easy to check in applications (see Proposition~\ref{P1} and~\eqref{spect} below)
we assume that
\begin{equation}\label{comp}
\text{$E_1$ is compactly embedded into $E_0$}\,.
\end{equation}
Then we can prove the following well-posedness result for problem~\eqref{E}:

\begin{thm}\label{T0}
Let $T>0$, $a\in \R$, and $b\in C^1([0,T],\R)$ with $b(0)\not=0$ be given. Suppose \eqref{Gen}, \eqref{A} with $\gamma\in (0,1]$, and suppose \eqref{comp}. Moreover, assume that the kernel of 
$\Phi_T \in \mathcal{L}(E_0,E_1)$ defined in~\eqref{Phi}
is trivial, that is,
\begin{equation}\label{spect}
\mathrm{ker}(\Phi_T)=\{0\}\,.
\end{equation}
Then, there is $m(T)>0$ such that, for each $M\in E_1$ with $\|M\|_1\le m(T)$, problem~\eqref{E} has a unique strong solution
$$
u\in C^1\big((0,T],E_0\big)\cap C\big((0,T], E_1\big)\cap C\big([0,T],E_0\big)\cap C_\theta\big((0,T], E_\theta\big)\,.
$$
\end{thm}

 Before continuing let us remark the following about the linear case when $f$ is independent of $u$:

\begin{rems}\label{R2}
{\bf(a)} In the linear case that $f$ is independent of $u$, there is no smallness condition required for $M\in E_1$ in order to obtain the existence of a unique solution to problem~\eqref{E} since \eqref{ulin} determines the initial state $u(0)$ uniquely for each $M\in E_1$. We therefore provide an alternative proof of \cite[Theorem~1]{Dokuchaev} (in fact, we prove a more general result). Also note that parabolic regularity enforces $M\in E_1$ since $\Phi_T$ maps into $E_1$ (see Proposition~\ref{P1}).\vspace{2mm}

{\bf(b)}  In \cite{Dokuchaev}, also conditions of the form
$$
a u(T)+b\int_0^\ve u(t)\,\rd t=M
$$
with $0<\ve<T$ are  considered, for which a similar result as in Theorem~\ref{T0} can be obtained along the same lines.\vspace{2mm}
\end{rems}

Let us also comment on the premises of Theorem~\ref{T0}.

\begin{rem}\label{R11} 
The condition $b(0)\not=0$ is also imposed in~\cite{Dokuchaev} for the corresponding linear problem (see Condition~1 therein)  and guarantees that certain information on the initial state is contained in the weighted time-average. Together with condition~\eqref{spect}  it implies that $\Phi_T\in \mathcal{L}is(E_0,E_1)$, see Proposition~\ref{P1}. If $A$ represents a symmetric,  elliptic differential operator, condition~\eqref{spect}  is not difficult to verify, see Corollary~\ref{C1} and Corollary~\ref{C2} in Section~\ref{Sec6}. 

In the particular case that $a=0$ and $b=const\not=0$, condition~\eqref{spect} is implied by
\begin{equation*}%\label{spectE100}
\sigma(A)\cap \frac{2\pi i}{T}\mathbb{Z}=\emptyset
\end{equation*}
owing to the fact that 
$$
A\Phi_T z=b\int_0^TAe^{tA}z\,\rd t=b\big(e^{TA}-1\big)z\,,\quad z\in E_0\,,
$$
and $\sigma\big(e^{TA}\big)\setminus\{0\}=e^{T\sigma(A)}$ by the spectral mapping theorem~\cite[V.Corollary~2.10]{EngelNagel}.
\end{rem}

As pointed out in the introduction, our result also applies to problem~\eqref{E100} and problem~\eqref{E200}. The statements are similar, we thus combine them in one theorem:

\begin{thm}\label{TE100}
Assume \eqref{Gen} and \eqref{comp}.
Let $T>0$ and let $f$ satisfy~\eqref{A}. \vspace{2mm} 

\begin{itemize}
\item[{\bf (i)}] For problem~\eqref{E100} assume $b\in \R\setminus\{0\}$ and the spectral condition
\begin{equation}\label{spectE100x}
1\notin b\, e^{T\sigma(A)}\,.
\end{equation}

\item[{\bf (ii)}] For problem~\eqref{E200} assume $b\in C([0,T],\R)$ with $b\not=0$ and the spectral condition
\begin{equation}\label{spectE200}
-1\notin\sigma_p\left(\int_0^T b(t)\,e^{tA}\,\rd t\right)\,.
\end{equation}
\end{itemize}
Then, there is $m(T)>0$ such that, for each $M\in E_0$ with $\|M\|_0\le m(T)$,  problem~\eqref{E100} respectively problem~\eqref{E200} has a unique mild  solution
$$
u\in C\big([0,T],E_0\big)\cap C_\theta\big((0,T], E_\theta\big)\,.
$$
In addition, if  $\gamma>0$  or if $f$ satisfies~\eqref{aaa}, then $$u\in C^1\big((0,T],E_0\big)\cap C\big((0,T], E_1\big)$$ is a strong solution.
\end{thm}

\begin{rem}\label{R2x}
The spectral condition~\eqref{spectE200} is satisfied whenever the spectral radius of the operator $\int_0^T b(t)e^{tA}\,\rd t\in \mathcal{K}(E_0)$
 is strictly less than 1. For instance, this is the case if $T>0$ is small enough. For concrete examples not assuming $T>0$ to be small see Corollary~\ref{C1} and Corollary~\ref{C2} in Section~\ref{Sec6}. 
\end{rem}

The outline of this paper is as follows. In Section~\ref{Sec3} we prove Theorem~\ref{T1} by a fixed point argument and thus   establish the well-posedness of the general problem~\eqref{EE}. Subsequently, in Section~\ref{Sec4} we demonstrate how to apply these findings to problem~\eqref{E} by deriving the invertibility of the operator $\Phi_T$  defined in~\eqref{Phi}. This yields then Theorem~\ref{T0}. In Section~\ref{Sec5} we show how to deduce Theorem~\ref{TE100} by applying Theorem~\ref{T1} to  problems~\eqref{E100} and~\eqref{E200}. We then consider more concrete situations. We provide in Section~\ref{Sec6} examples for operators $A$ satisfying the spectral conditions imposed in the previous theorems. In particular, examples include  symmetric uniformly elliptic differential operators of second or fourth order, see Corollary~\ref{C1} respectively Corollary~\ref{C2}. In Section~\ref{Sec7} we give examples for nonlinearities $f$ that obey assumption~\eqref{A}. Finally, in Section~\ref{Sec8} we consider the semilinear version of problem~\eqref{DD} and illustrate how to apply our result in such a setting.

%%%%%%%%%%%%%%%%%%%%%%%%%%%%%%%%%%%%%%%%%%%%%%%%
%%%%%%%%%%%%%%%%%%%%%%%%%%%%%%%%%%%%%%%%%%%%%%%%
\section{Proof of Theorem~\ref{T1} and Corollary~\ref{C11}}\label{Sec3}
%%%%%%%%%%%%%%%%%%%%%%%%%%%%%%%%%%%%%%%%%%%%%%%%
%%%%%%%%%%%%%%%%%%%%%%%%%%%%%%%%%%%%%%%%%%%%%%%%

We focus on the general problem~\eqref{EE}; that is, on
$$
u'=Au+f(u,t)\,,\quad t\in(0,T]\,,\qquad u(0)=\Sigma_T(u)\,.
$$
First, we prove Theorem~\ref{T1} and then derive Corollary~\ref{C11} as a consequence.

\subsection*{Proof of Theorem~\ref{T1}} In order to prove Theorem~\ref{T1},  we assume~\eqref{Gen}, \eqref{A}, and \eqref{A3}.
We introduce the complete metric space
$$
\mathcal{V}_L:=\big\{v\in C_{\theta}\big((0,T],E_\theta\big)\cap C\big([0,T],E_0\big)\,;\, \|v\|_{C_{\theta}((0,T],E_\theta)}+ \|v\|_{C([0,T],E_0)}\le L\big\}
$$
with $L\in (0,1)$ to be determined later and aim at deriving a fixed point of the mapping $\Lambda$ defined~by
\begin{equation*}
\Lambda(u)(t):=e^{tA}\Sigma_T(u)+\int_0^t e^{(t-s)A} f(u,s)\,\rd s\,,\qquad t\in [0,T]\,,\quad u\in\mathcal{V}_L\,.
\end{equation*}
Such a fixed point corresponds to a (mild) solution to~\eqref{EE}. The existence of a fixed point will be guaranteed by Banach's fixed point theorem once we have established that $\Lambda:\mathcal{V}_L\to \mathcal{V}_L$ is a contraction (for $L$ chosen suitably).

We first recall  from~\cite[II.~Lemma~5.1.3]{LQPP} the estimates 
\begin{equation}\label{b0}
\|e^{tA}\|_{\mathcal{L}(E_\vartheta)}+  t^{\alpha-\beta_0}\,\|e^{tA}\|_{\mathcal{L}(E_\beta,E_\alpha)} 
 \le \omega(T)
\,, \qquad 0\le  t\le T\,, 
\end{equation}
 for $\vartheta\in[0,1]$ and $0\le \beta_0\le \beta\leq \alpha\le 1$ 
 with $\beta_0<\beta$ if $0<\beta<\alpha<1$,  where $\omega(T)>1$ depends also on these parameters. We put $\gamma_0:=0$ if $\gamma=0$ and otherwise chose $\gamma_0\in \big(0,\min\lbrace \gamma,\theta \rbrace\big)$ arbitrarily.

Let $u\in \mathcal{V}_L$ and $0<t\le T$. Since~\eqref{A2}  and \eqref{A3} imply
\begin{equation}\label{AAA2}
\|f(u,t)\|_\gamma\le c_T\,L^{\ell+1}\, t^{-\nu}% + \|f(0,t)\|_\gamma
\end{equation}
and
\begin{equation}\label{AAA3}
\|\Sigma_T(u)\|_0\le c_T\,L^{\ell+1}+\|\Sigma_T(0)\|_0\,,
\end{equation}
respectively, we readily obtain from $\Sigma_T(u)\in E_0$ that $\Lambda(u)\in C\big([0,T],E_0\big)$. Moreover, using~\eqref{b0}, $\nu<1$, and denoting by $c_\gamma$ the norm of the embedding $E_\gamma\hookrightarrow E_0$ we deduce
\begin{align} 
\|\Lambda(u)(t)\|_0&\le \|e^{tA}\|_{\mathcal{L}(E_0)}\,\|\Sigma_T(u)\|_0+c_\gamma\int_0^t \|e^{(t-s)A}\|_{\mathcal{L}(E_0)}\, \|f(u,s)\|_\gamma\,\rd s\nonumber\\
&\le \omega(T)\,\big\{ c_T\, L^{\ell+1} +\|\Sigma_T(0)\|_0\big\} +c_\gamma\,\omega(T) c_T\,L^{\ell+1}\,T^{1-\nu}\,.\label{e1}%\\
%&\qquad + \omega(T)\|f(0,\cdot)\|_{L_1((0,t),E_\gamma)}\,.
\end{align}
Since, due to~\eqref{b0},
\begin{align*}
t^\theta\| e^{tA}\Sigma_T(u)\|_\theta &\le t^\theta\| e^{tA}\|_{\mathcal{L}(E_0,E_\theta)}\,\|\Sigma_T(u)-u_0\|_0+t^\theta \| e^{tA}\|_{\mathcal{L}(E_\theta)}\,\|u_0\|_\theta\\
&\le \omega(T)\big( \|\Sigma_T(u)-u_0\|_0+t^\theta\|u_0\|_\theta\big)
\end{align*}
for $t\in (0,T]$ and each $u_0\in E_\theta$, we may use the density of $E_\theta$ in $E_0$ and $\theta\in (0,1]$ to derive that
\begin{align}\label{e1A}
\big[t\mapsto e^{tA}\Sigma_T(u)\big]\in C_{\theta}\big((0,T],E_\theta\big)
\end{align}
with
$$
\|e^{\cdot A}\Sigma_T(u)\|_{C_{\theta}((0,T],E_\theta)}\le \omega(T)\,\|\Sigma_T(u)\|_0\le  \omega(T) \big\{c_T\,L^{\ell+1}+\,\|\Sigma_T(0)\|_0\big\}\,,
$$
where we used~\eqref{AAA3} for the second inequality.
Recall that $\theta<1+\gamma_0$ and $\nu<1$ in view of~\eqref{A1}. Hence, using~\eqref{b0} and~\eqref{AAA2}, we estimate
\begin{align*} 
\|\Lambda(u)(t)-e^{tA}\Sigma_T(u)\|_{\theta}&\le \int_0^t \|e^{(t-s)A}\|_{\mathcal{L}(E_\gamma,E_\theta)}\,\|f(u,s)\|_\gamma\,\rd s\\
&\le  \omega(T)\,c_T\,L^{\ell+1}\int_0^t (t-s)^{\gamma_0-\theta}\, s^{-\nu}\,\rd s\\
%&\qquad + \omega(T)\int_0^t (t-s)^{\gamma_0-\theta}\, s^{-\nu} \,\rd s\,\|f(0,\cdot)\|_{C_{\nu}((0,t],E_\gamma)}\\
& =  \omega(T)\, c_T\,L^{\ell+1} \, t^{1+\gamma_0-\theta-\nu} \,\mathsf{B}(1+\gamma_0-\theta,1-\nu)
\end{align*}
with $\mathsf{B}$ denoting the Beta function. Consequently, since $1+\gamma_0-\nu>0$ and using~\eqref{e1A}, we derive that $\Lambda(u)\in C_{\theta}\big((0,T],E_\theta\big)$ with 
\begin{align}
\|\Lambda(u)\|_{C_{\theta}((0,T],E_\theta)}&\le \omega(T)\big\{c_T\, L^{\ell+1}+\,\|\Sigma_T(0)\|_0\big\}\nonumber\\
&\quad +  \omega(T) c_T\,L^{\ell+1} \, T^{1+\gamma_0-\nu} \,\mathsf{B}(1+\gamma_0-\theta,1-\nu)\,.\label{e2}
\end{align}
Likewise, for  $u,v\in \mathcal{V}_L$, $0<t\le T$, we obtain from~\eqref{A} and~\eqref{b0}
\begin{align*} 
\|\Lambda(u)(t)-\Lambda(v)(t)\|_{\theta}&\le \|e^{tA}\|_{\mathcal{L}(E_0,E_\theta)}\,\| \Sigma_T(u)-\Sigma_T(v)\|_0
\\
&\quad+ \int_0^t \|e^{(t-s)A}\|_{\mathcal{L}(E_\gamma,E_\theta)}\, 
 \|f(u,s)-f(v,s)\|_\gamma\,\rd s\\
&\le \omega(T) t^{-\theta}\, c_T\,L^{\ell}\, \|u-v\|_{C_{\theta}((0,T],E_\theta)}
\\
&\quad + \omega(T)\, c_T\,L^{\ell}  \int_0^t (t-s)^{\gamma_0-\theta}\, s^{-\nu}\,\rd s\, \|u-v\|_{C_{\theta}((0,T],E_\theta)}
\end{align*}
and therefore
%\begin{align} 
%\|\Lambda(u)(t)-\Lambda(v)(t)\|_{0}&\le \omega(T) \, c_T(L)\,L^{\ell}\, \|u-v\|_{C_{\theta}((0,T],E_\theta)}
%\nonumber \\
%&\quad + \omega(T)\, c_T(L)\,L^{\ell}\, T^{1+\gamma_0-\nu} \,\mathsf{B}(1+\gamma_0,1-\nu)\, %\|u-v\|_{C_{\theta}((0,T],E_\theta)}\label{e3}
%\end{align}
%and
\begin{align} 
\|\Lambda(u)&-\Lambda(v)\|_{C_\theta((0,T],E_\theta)}\nonumber \\
&\le   \omega(T)\, c_T\, L^{\ell}\,\big[1+T^{1+\gamma_0-\nu} \,\mathsf{B}(1+\gamma_0-\theta,1-\nu)\big]\, \|u-v\|_{C_{\theta}((0,T],E_\theta)}\,.\label{e3a}
\end{align}
A similar  but simpler estimate based on \eqref{A} and~\eqref{b0} also yields
\begin{align} 
\|\Lambda(u)-\Lambda(v)\|_{C([0,T],E_0)} \le   \omega(T)\, c_T\, L^{\ell}\,\Big[1+\frac{c_\gamma\,T^{1-\nu}}{1-\nu}\Big]\, \|u-v\|_{C_{\theta}((0,T],E_\theta)}\,.\label{e3}
\end{align}

Since $\ell>0$, we deduce from~\eqref{e1}-\eqref{e3} that there is $L=L(T)\in(0,1)$ such that $\Lambda:\mathcal{V}_L\to \mathcal{V}_L$ is a contraction whenever $\Vert\Sigma_T(0)\Vert_0\le m(T):=L/(4\omega(T))$. Hence, Banach's fixed point theorem ensures the existence of a unique $u\in \mathcal{V}_L$ with $\Lambda(u)=u$. Clearly, $u$ is a mild solution to~\eqref{EE}. 

 Finally, we show that $u$ is a strong solution if $\gamma>0$. For each $\ve\in (0,T)$, the function  $u_\ve:=u(\cdot+\ve)\in C([0,T-\ve],E_\theta)$ is a mild solution to the linear Cauchy problem
$$
u_\ve'=Au_\ve +f_\ve(t)\,,\quad t\in [0,T-\ve]\,,\qquad u_\ve(0)=u(\ve)\in E_\theta\,,
$$
with $f_\ve:=f(u,\cdot+\ve)\in C([0,T-\ve],E_\gamma)$. Since $\gamma>0$, \cite[II.Theorem~1.2.2]{LQPP} implies that
$$
u_\ve\in C^1\big((0,T-\ve],E_0\big)\cap C\big((0,T-\ve], E_1\big)
$$
is a strong solution. Since $\ve> 0$ was arbitrary, we deduce that $u$ is a strong solution with the stated regularity, hence the proof of Theorem~\ref{T1} is complete.\qed

\subsection*{Proof of Corollary~\ref{C11}} Let now assumption~\eqref{Gen},~\eqref{A3}, and \eqref{aaa} be satisfied. Setting
$$
f(v)(t):= f(v(t))\,,\qquad t\in (0,T]\,,\quad v\in C_{\theta}\big((0,T],E_\theta\big)\,,
$$
and noticing that
$$
t^\theta\|v(t)\|_\theta\le \|v\|_{C_{\theta}((0,t],E_\theta)}\,,\quad t\in (0,T]\,,\qquad \lim_{t\to 0^+}\|v\|_{C_{\theta}((0,t],E_\theta)}= 0\,,
$$
for $v\in C_{\theta}\big((0,T],E_\theta\big)$, it readily follows from~\eqref{aaa} that
$$
f:C_{\theta}\big((0,T],E_\theta\big)\to C_{\nu}\big((0,T],E_0\big)
$$ 
 satisfies~\eqref{A2} with $\nu=\theta(\ell+1)\in [0,1)$ and $\gamma=0$. Consequently, Theorem~\ref{T1} ensures the existence of a unique mild solution $u\in C_{\theta}\big((0,T],E_\theta\big)\cap C\big([0,T],E_0\big)$ to~\eqref{EE}.

It remains to show that it is also a strong solution. Since $\theta<1$ by assumption~\eqref{aaa}, we can choose $\xi\in (\theta,1)$. As above, for each $\ve\in (0,T)$, the function  $u_\ve:=u(\cdot+\ve)\in C([0,T-\ve],E_\theta)$ is a mild solution to the linear Cauchy problem
$$
u_\ve'=Au_\ve +f_\ve(t)\,,\quad t\in [0,T-\ve]\,,\qquad u_\ve(0)=u(\ve)\in E_\theta\,,
$$
with $f_\ve:=f(u_\ve)\in C([0,T-\ve],E_0)$. In fact, \cite[II.Theorem~5.4.1]{LQPP} ensures that we may assume without loss of generality that  $u_\ve(0)\in E_\xi$ by making $\ve>0$ smaller so that we infer from \cite[II.Theorem~5.3.1]{LQPP} that $u_\ve\in C^{\xi-\theta}\big([0,T-\ve],E_\theta)$. Therefore, $f_\ve\in C^{\xi-\theta}\big([0,T-\ve],E_0)$ according to~\eqref{aaa} and we may thus apply \cite[II.Theorem~1.2.1]{LQPP} to conclude that $u_\ve$ is a strong solution to the linear Cauchy problem. Since $\ve> 0$ was arbitrary, also $u$ is a strong solution, and we deduce Corollary~\ref{C11}.\qed

%%%%%%%%%%%%%%%%%%%%%%%%%%%%%%%%%%%%%%%%%%%%%%%%
%%%%%%%%%%%%%%%%%%%%%%%%%%%%%%%%%%%%%%%%%%%%%%%%
\section{Proof of Theorem~\ref{T0}}\label{Sec4}
%%%%%%%%%%%%%%%%%%%%%%%%%%%%%%%%%%%%%%%%%%%%%%%%
%%%%%%%%%%%%%%%%%%%%%%%%%%%%%%%%%%%%%%%%%%%%%%%%
We shall apply Theorem~\ref{T1} to problem~\eqref{E} and thus prove Theorem~\ref{T0}. We have already shown that problem~\eqref{E} is for a given $M\in E_1$ (formally) equivalent to 
$$
u'=Au+f(u,t)\,,\quad t\in(0,T]\,,\qquad u(0)=\Sigma_T(u)\,,
$$
where (see~\eqref{u0})
\begin{equation*}
\Sigma_T(u)=\Phi_T^{-1}\big(M-\Psi_T(f(u))\big)
\end{equation*}
with operators  $\Phi_T$ and $\Psi_T$ given by (see~\eqref{Phi} and \eqref{Psi})
\begin{equation*}
\Phi_T=a\,e^{TA}+\int_0^Tb(t)\,e^{tA}\,\rd t
\end{equation*}
respectively
\begin{equation*}
\Psi_T g=a\int_0^T e^{(T-s)A} g(s)\,\rd s+ \int_0^Tb(t)\int_0^t e^{(t-s)A} g(s)\,\rd s\,\rd t\,.
\end{equation*}
In order  to apply Theorem~\ref{T1} we have to make this more rigorous. We first note that $\Phi_T$ is an isomorphism:

\begin{prop}\label{P1}
Suppose \eqref{Gen} and \eqref{comp}. Moreover, let $a\in \R$, $b\in C^1([0,T],\R)$ with $b(0)\not= 0$ and assume~\eqref{spect}. 
Then
$\Phi_T\in \mathcal{L}is(E_0,E_1)$.
\end{prop}

\begin{proof}
Since integration by parts yields
\begin{align*}
A\Phi_T x&=aAe^{TA}x+\int_0^T b(t)\,A\,e^{tA}x\,\rd t=\big(aA+b(T)\big)e^{TA}x-\int_0^Tb'(t)\,e^{tA}x\,\rd t-b(0)x
\end{align*}
for $x\in E_0$,  we deduce $A\Phi_T\in \mathcal{L}(E_0)$ and thus $\Phi_T\in \mathcal{L}(E_0,E_1)$.
Fix $\mu\in\rho(A)$ in the resolvent set. Then, it follows from~\eqref{b0} that
$$
K_\mu:=\left(\big(aA+b(T)\big)e^{TA}-\int_0^Tb'(t)\,e^{tA}\,\rd t-\mu\Phi_T\right)\in  \mathcal{L}(E_0,E_\vartheta)\subset\mathcal{K}(E_0)
$$
for any $ \vartheta\in (0,1)$, where we use that $E_\vartheta$ embeds compactly into $E_0$ by assumption~\eqref{comp} and e.g.~\cite[I.Theorem~2.1.1]{LQPP}. Since $b(0)\not=0$, the Riesz-Schauder theorem now ensures that 
$$
(A-\mu)\Phi_T=K_\mu-b(0) \in\mathcal{L}(E_0)
$$ 
is a Fredholm operator of index zero. In fact, condition~\eqref{spect} implies that $\mathrm{ker}\big((A-\mu)\Phi_T\big)=\{0\}$, and hence $(A-\mu)\Phi_T$ is an isomorphism on $E_0$. Consequently 
$$
\Phi_T=(A-\mu)^{-1}(A-\mu)\Phi_T\in \mathcal{L}is(E_0,E_1)
$$
as claimed.
\end{proof}

For $\Psi_T$  we note:

\begin{lem}\label{L1}
Assume~\eqref{Gen} and let $a\in \R$, $b\in C([0,T],\R)$, $\gamma\in (0,1]$, and $\nu\in [0,1)$. Then $\Psi_T\in\mathcal{L}\big(C_\nu((0,T],E_\gamma), E_1\big)$.% and there is $c(T,\gamma,\nu)>0$ with
%\begin{align*}
%\|\Psi_T g\|_1&\le c(T,\gamma,\nu)\ (a+\|b\|_\infty T)\, \|g\|_{C_\nu((0,T],E_\gamma)}
%\end{align*}
%for $g\in C_\nu\big((0,T],E_\gamma\big)$.
\end{lem}

\begin{proof}
It readily follows from the definition of $\Psi_T$, see~\eqref{Psi}, and~\eqref{b0}  that
\begin{align}
\|\Psi_Tg\|_1&\le \vert a\vert\int_0^T\|e^{(T-s)A}\|_{\mathcal{L}(E_\gamma,E_1)}\, \|g(s)\|_{\gamma}\,\rd s +\int_0^T\vert b(t)\vert\int_0^t\|e^{(t-s)A}\|_{\mathcal{L}(E_\gamma,E_1)}\, \|g(s)\|_{\gamma}\,\rd s\,\rd t\nonumber\\
&\le \vert a\vert\omega(T)\int_0^T (T-s)^{\gamma-1}\, s^{-\nu} \,\rd s\, \|g\|_{C_\nu((0,T],E_\gamma)}\nonumber\\
&\quad +\|b\|_\infty\omega(T)\int_0^T\int_0^t(t-s)^{\gamma-1}\, s^{-\nu}\,\rd s\,\rd t\, \|g\|_{C_\nu((0,T],E_\gamma)}\nonumber\\
&=\omega(T) T^{\gamma-\nu}\left[\vert a\vert+\frac{\|b\|_\infty T}{1+\gamma-\nu}\right]\,\mathsf{B}(\gamma,1-\nu)\, \|g\|_{C_\nu((0,T],E_\gamma)}\label{h1}
\end{align}
for $g\in C_\nu\big((0,T],E_\gamma\big)$. This yields the claim.
\end{proof}

\subsection*{Proof of Theorem~\ref{T0}}
To finish off the proof of Theorem~\ref{T0} suppose the assumptions stated therein. In particular, $f$ satisfies~\eqref{A} with $\gamma>0$. This together with Proposition~\ref{P1} and Lemma~\ref{L1} imply for any given $M\in E_1$ that $\Sigma_T$, given by~\eqref{u0}, is a mapping
$$
\Sigma_T: C_{\theta}\big((0,T],E_\theta\big)\to E_0
$$
such that there is $c_T>0$ with
\begin{equation*}
\|\Sigma_T(u)-\Sigma_T(v)\|_0\le \|\Phi_T^{-1}\|_{\mathcal{L}(E_1,E_0)}\,\left\| \Psi_T\big(f(u)-f(v)\big)\right\|_{1}\le c_T\,L^{\ell}\, \|u-v\|_{C_{\theta}((0,T],E_\theta)}
\end{equation*}
for $L\in(0,1)$ and $u,v\in C_{\theta}((0,T],E_\theta)$ with 
$$
\|u\|_{C_{\theta}((0,T],E_\theta)}+\|v\|_{C_{\theta}((0,T],E_\theta)}\le L\,,\qquad \sup_{t\in (0,T]}\big(\|u(t)\|_0+\|v(t)\|_0\big)\le L\,.
$$
Therefore, $\Sigma_T$ satisfies~\eqref{A3} and 
$$
\|\Sigma_T(0)\|_0\le \|\Phi_T^{-1}\|_{\mathcal{L}(E_1,E_0)}\,\left\| M\right\|_{1}\,,
$$ 
where we recall that $f(0)=0$.
Theorem~\ref{T0} is now a consequence of Theorem~\ref{T1}.\qed

\begin{rem}\label{R2xx}
Let  $a=0$ and $\gamma>\nu$ in \eqref{A1}. Then, it follows from the proof of Lemma~\ref{L1} (see~\eqref{h1}) that 
$$
\|\Psi_T\|_{\mathcal{L}(C_\nu((0,T],E_\gamma), E_1)}\le c T^{1+\gamma-\nu}\,.
$$ 
Therefore, if 
\begin{equation}\label{pi}
\|\Phi_T^{-1}\|_{\mathcal{L}(E_1,E_0)}\le \frac{c}{T}\,,
\end{equation}
then, for any given $m_0>0$, one can ensure that
$$
\|\Sigma_T(0)\|_0=\|\Phi_T^{-1}(M-\Psi_T(f(0)))\|_0\le \frac{c}{T}\|M\|_1+ c T^{\gamma-\nu}\|f(0)\|_{C_\nu((0,T],E_\gamma)}\le m_0\,,
$$
by taking $T\in (0,T_0)$ and $\|M\|_1$ sufficiently small. This yields the situation of Remark~\ref{R1}. Consequently, provided $a=0$, $\gamma>\nu$, and \eqref{pi} holds,   one can drop the assumption $f(0)=0$ in \eqref{A2} and still show well-posedness of problem~\eqref{E}  as stated in Theorem~\ref{T0} for small values of $T$.

For example, condition~\eqref{pi} can be verified for operators in Hilbert spaces, see Remark~\ref{R3} below.
\end{rem}

%%%%%%%%%%%%%%%%%%%%%%%%%%%%%%%%%%%%%%%%%%%%%%%%
%%%%%%%%%%%%%%%%%%%%%%%%%%%%%%%%%%%%%%%%%%%%%%%%

\section{Proof of Theorem~\ref{TE100}}\label{Sec5}

%%%%%%%%%%%%%%%%%%%%%%%%%%%%%%%%%%%%%%%%%%%%%%%%
%%%%%%%%%%%%%%%%%%%%%%%%%%%%%%%%%%%%%%%%%%%%%%%%

In order to prove Theorem~\ref{TE100} we assume the conditions stated therein and  proceed similarly as in the previous section, distinguishing between problem~\eqref{E100} and problem~\eqref{E200}.

\subsection*{Proof for Problem~\eqref{E100}} Consider a mild solution  $u$ to~\eqref{E100} for a given $M\in E_0$; that is, a mild solution to
$$
u'=Au+f(u,t)\,,\quad t\in(0,T]\,,\qquad u(0)-bu(T)=M\,.
$$
Plugging the formula \eqref{vdk} into $u(0)-bu(T)=M$ yields the condition
$$
u(0)=\big(1-be^{TA}\big)^{-1}\left(M+b\int_0^Te^{(T-s)A}f(u,s)\,\rd s\right)\,.
$$
That is, in this case we have
\begin{equation}\label{sigma1}
\Sigma_T(u):=K_T^{-1}\big(M+Z_T(f(u))\big)
\end{equation}
with
$$
K_T:=1-be^{TA}\,,\qquad Z_Tg:=b\int_0^Te^{(T-s)A}g(s)\,\rd s\,.
$$
Note that  $K_T$ is a zero index Fredholm operator since $be^{TA}\in\mathcal{K}(E_0)$ by~\eqref{Gen} and~\eqref{comp}, hence assumption~\eqref{spectE100x} and the spectral mapping theorem~\cite[IV.Corollary~3.12]{EngelNagel} ensure   $K_T^{-1}\in\mathcal{L}(E_0)$. Moreover, as in Lemma~\ref{L1} it follows from~\eqref{b0} that
$Z_T\in\mathcal{L}\big(C_\nu((0,T],E_\gamma), E_0\big)$ with
\begin{align*}
\|Z_T g\|_0&\le \frac{\omega(T)T^{1-\nu}c_\gamma \vert b\vert}{1-\nu}\, \|g\|_{C_\nu((0,T],E_\gamma)}
\end{align*}
for $\gamma\in [0,1]$ and $\nu\in [0,1)$. As in the previous section this together with $f(0)=0$ implies that $\Sigma_T$ defined in \eqref{sigma1}
satisfies~\eqref{A3} and
$$
\Sigma_T(0)=K_T^{-1}M\,.
$$ 
Theorem~\ref{T1} and Corollary~\ref{C11} now yield Theorem~\ref{TE100} for problem~\eqref{E100}.\qed

%%%%%%%%%%%%%%%%%%%%%%%%%%%%%%%%%%%%%%%%%%%%%%%%
%%%%%%%%%%%%%%%%%%%%%%%%%%%%%%%%%%%%%%%%%%%%%%%%

%\section{Proof of Theorem~\ref{TE200}}

%%%%%%%%%%%%%%%%%%%%%%%%%%%%%%%%%%%%%%%%%%%%%%%%
%%%%%%%%%%%%%%%%%%%%%%%%%%%%%%%%%%%%%%%%%%%%%%%%

\subsection*{Proof for Problem~\eqref{E200}}  If $u$ is a mild solution to~\eqref{E200} for a given $M\in E_0$; that is, to
$$
u'=Au+f(u,t)\,,\quad t\in(0,T]\,,\qquad u(0)+\int_0^T b(t) u(t)\,\rd t=M\,,
$$
we obtain the condition
\begin{equation}\label{sigma2}
u(0)=\Sigma_T(u):=\big(1+\Phi_T^0\big)^{-1}\left(M-\Psi_T^0(f(u))\right)
\end{equation}
with
$$
\Phi_T^0:=\int_0^T b(t)\,e^{tA}\,\rd t\,,\qquad \Psi_T^0g:=\int_0^T b(t)\int_0^t e^{(t-s)A}g(s)\,\rd s\,\rd t\,.
$$
Note that $\Phi_T^0\in\mathcal{K}(E_0)$, hence $1+\Phi_T^0$ is a zero index Fredholm operator and assumption~\eqref{spectE200} then yields that $\big(1+\Phi_T^0\big)^{-1}\in\mathcal{L}(E_0)$. The proof of Lemma~\ref{L1} entails that
$\Psi_T^0\in\mathcal{L}\big(C_\nu((0,T],E_\gamma), E_0\big)$ for $\gamma\in [0,1]$ and $\nu\in [0,1)$.
Thus, $\Sigma_T$ defined in \eqref{sigma2}
satisfies~\eqref{A3} with
$$
\Sigma_T(0)=\big(1+\Phi_T^0\big)^{-1}M\,,
$$ 
so that Theorem~\ref{T1} (and Corollary~\ref{C11}) implies Theorem~\ref{TE100} for problem~\eqref{E200}.\qed

\section{Examples for the Operator $A$}\label{Sec6}

We provide examples for operators $A$ satisfying the assumptions of the previous theorems, in particular, satisfying condition~\eqref{spect} or~\eqref{spectE100x} or~\eqref{spectE200}. We first consider general self-adjoint operators in Hilbert spaces and then focus on symmetric elliptic differential operators in $L_p$-spaces.%\\

%In the following, assume that $a\in \R$ and $b\in C([0,T],\R)$.

\subsection*{The Hilbert Space Case}  In  the Hilbert space case one may use Fourier series so that the spectral conditions take a particular form as has been observed in \cite{Dokuchaev}.
Let
\begin{subequations}\label{As}
\begin{equation}
\begin{split}
&\text{$A\le \alpha <\infty$ be  a closed, densely defined, self-adjoint operator}\\
&\text{on the Hilbert space $\big(E_0,(\cdot\vert\cdot)\big)$ with  compact resolvent},
\end{split}
\end{equation}
where $A\le \alpha $ means that $(Ax\vert x)\le \alpha \|x\|^2$ for $x\in E_0$. Then the spectrum of $A$ consists of countably many real eigenvalues 
\begin{equation}
\alpha\ge\lambda_1\ge \cdots\ge \lambda_j\to-\infty
\end{equation} 
\end{subequations}
(according to multiplicity) and the corresponding normalized eigenvectors $(\phi_j)_{j\ge 1}$ belong to
$$
E_1=\mathrm{dom}(A)=\Big\{x\in E_0\,;\, \sum_j \vert\lambda_j\vert^2\vert(x\vert\phi_j)\vert^2<\infty\Big\}
$$ 
and build an orthonormal basis in $E_0$.  Fourier series expansion yields for the operator $\Phi_T$ introduced in~\eqref{Phi} the representation
$$
\Phi_T x=\sum_j \left(ae^{T\lambda_j}+\int_0^Tb(t)e^{t\lambda_j}\,\rd t\right)\, \big(x\vert \phi_j\big)\,\phi_j\,,\qquad x\in E_0\,,
$$
and similarly for $\Phi_T^0$ from~\eqref{sigma2}
from which we readily derive the kernel and spectral conditions:

\begin{prop}\label{PP1}
Assume~\eqref{As}. Then we have:\vspace{2mm}

{\bf (a)} Let $a\in \R$ and $b\in C([0,T],\R)$. Condition~\eqref{spect} is equivalent to
\begin{equation*}
-a\not= \int_0^Tb(t)\,e^{-\lambda_j(T-t)}\,\rd t\,,\quad j\in\N\,.
\end{equation*}
In particular,  if $b\ge 0$, then \eqref{spect} is satisfied whenever $a\ge 0$ and $b(0)>0$ or $\alpha\le  0$ and $\displaystyle\int_0^T b(t)\,\rd t> -a$.\vspace{2mm}
%\begin{equation}\label{spect2X}
%a\ge 0 \quad \text{ and }\quad b\ge 0 \ \text{ with }\ b(0)>0\,.
%\end{equation}

{\bf (b)}  Let $b\in \R\setminus\{0\}$. Condition~\eqref{spectE100x} is equivalent to
\begin{equation*} 
1\not= b\, e^{\lambda_j T}\,,\quad j\in\N\,.
\end{equation*}
In particular,  \eqref{spectE100x} is satisfied whenever $b\le 1$ and $\alpha<0$.\vspace{2mm}

{\bf (c)} Let $b\in C([0,T],\R)$.  Condition~\eqref{spectE200} is equivalent to
\begin{equation*} 
-1\not= \int_0^Tb(t)\,e^{-\lambda_j t}\,\rd t\,,\quad j\in\N\,.
\end{equation*}
In particular,  \eqref{spectE200} is satisfied whenever $b\ge 0$.
\end{prop}
 
Since the inverse of $\Phi_T$ can be computed explicitly in terms of Fourier series, one can estimate its norm in terms of $T>0$ and verify \eqref{pi}.

\begin{rem}\label{R3}
Assuming~\eqref{As}, $a=0$, and $b\equiv 1$, we have
$$
\Phi_T^{-1}x=\sum_j \mu_j(T) (x\vert\phi_j)\phi_j\,,\quad x\in E_1=\mathrm{dom}(A)\,,
$$
with
$$
\mu_j(T):=\left\{\begin{array}{ll}
\dfrac{\lambda_j}{e^{T\lambda_j}-1}\,, &\lambda_j\not= 0 \,,\\[3mm]
T^{-1}\,, &\lambda_j=0\,,
\end{array}\right.
$$
and therefore,
$$
 \|\Phi_T^{-1}\|_{\mathcal{L}(E_1,E_0)}\sim T^{-1}\,,\quad T\to 0\,,
$$
that is, \eqref{pi} is satisfied.
\end{rem}

\subsection*{Elliptic Differential Operators} To consider more specifically elliptic differential operators, let $\Omega\subset \R^n$ be  a bounded smooth domain with outer unit normal $\nu$. Consider functions
\begin{subequations}\label{A11}
\begin{equation}\label{g1}
c\in C(\bar\Omega,\R^+)\,,\qquad d\in C^1\big(\bar\Omega,\R_{sym}^{n\times n}\big)\,,\qquad d(x)\zeta\cdot\zeta\ge \underline{d}\, \vert\zeta\vert^2\,,\quad (\zeta,x)\in\R^n\times \bar\Omega\,,
\end{equation}
for some constant $\underline{d}>0$. Given $\delta\in\{0,1\}$ define the boundary operator
\begin{equation}
\mathcal{B}u:=(1-\delta)u+\delta \partial_\nu u
\end{equation}
(i.e. for $\delta=0$ and $\delta=1$ we consider Dirichlet respectively Neumann boundary conditions) and set
\begin{equation}
 W_{p,\mathcal{B}}^2(\Omega):=\big\{u\in W_{p}^2(\Omega)\,;\, \mathcal{B}u=0\ \text{on}\ \partial\Omega\big\}
\end{equation}
for $p\in (1,\infty)$. We then consider the uniformly elliptic second-order differential operator $A_p$ on~$L_p(\Omega)$ given by
\begin{equation}\label{Ap}
A_pu:=\mathrm{div}\big(d(x)\nabla u\big)-c(x) u\,,\quad u\in W_{p,\mathcal{B}}^2(\Omega)\,.
\end{equation}
\end{subequations}
Note that $E_1:=W_{p,\mathcal{B}}^2(\Omega)$ embeds compactly into $E_0:=L_p(\Omega)$, i.e. \eqref{comp} is satisfied.
We then obtain:

\begin{cor}\label{C1}
Assume~\eqref{A11}. For the operator $A_p\in \mathcal{H}\big(W_{p,\mathcal{B}}^2(\Omega),L_p(\Omega)\big)$ defined in~\eqref{Ap} with $p\in (1,\infty)$ we have:\vspace{2mm}

{\bf (a)}   Condition~\eqref{spect} is satisfied if~ $b\in C([0,T],\R^+)$ with $\displaystyle\int_0^T b(t)\,\rd t> -a$ and $p\in [2,\infty)$.\vspace{2mm}

{\bf (b)} Condition~\eqref{spectE100x} is satisfied if $b\le 1$ and if $\underline{c}:=\min_{\bar\Omega} c>0$ when $\delta=1$.\vspace{2mm}

{\bf (c)} Condition~\eqref{spectE200} is satisfied if $b\in C([0,T],\R^+)$.

\end{cor}

\begin{proof}
It is well known (e.g. see~\cite{Amann_Teubner}) that  $A_p\in \mathcal{H}\big(W_{p,\mathcal{B}}^2(\Omega),L_p(\Omega)\big)$ and that $A_2$ satisfies~\eqref{As} on the Hilbert space $L_2(\Omega)$ with $\alpha\le 0$. Moreover, the (point) spectra $\sigma(A_p)=\sigma(A_2)$ coincide~\cite{AmannIsrael}. We can now apply  Proposition~\ref{PP1} to the operator $A_2$.\vspace{2mm}

 {\bf (a)} The imposed assumptions  imply~\eqref{spect} for $A_2$ by Proposition~\ref{PP1}~{\bf (a)}; that is, $\mathrm{ker}\big(\Phi_T^{(2)}\big)=\{0\}$ for
\begin{equation*}
\Phi_T^{(q)}:=\left(a\,e^{TA_q}+\int_0^Tb(t)\,e^{tA_q}\,\rd t\right)\in \mathcal{L}\big(L_q(\Omega),W_{q,\mathcal{B}}^2(\Omega)\big)\,,\quad q\in [2,\infty)\,.
\end{equation*}
Since $e^{TA_2}\big\vert_{L_p(\Omega)}=e^{TA_p}$ for $p\ge 2$ (see e.g.~\cite{AmannIsrael}), we have $\Phi_T^{(p)}\subset \Phi_T^{(2)}$ and thus  $\mathrm{ker}\big(\Phi_T^{(p)}\big)=\{0\}$. That is,~$A_p$ satisfies~\eqref{spect}.\vspace{2mm}

{\bf (b)} Since  $\sigma(A_p)=\sigma(A_2)\subset (-\infty,\alpha]$ with $\alpha<0$ (due to assumption~\eqref{g1} and the fact that $\underline{c}>0$ if $\delta=1$), the assertion follows from Proposition~\ref{PP1}~{\bf (b)}.\vspace{2mm}

{\bf (c)} Since $-1$ is an eigenvalue of $\int_0^T b(t)\,e^{tA_p}\,\rd t$ if and only if it is an eigenvalue of $\int_0^T b(t)\,e^{tA_2}\,\rd t$ by parabolic regularity theory, the assertion follows from Proposition~\ref{PP1}~{\bf (c)}.
\end{proof}

Corollary~\ref{C1} extends to the interpolation-extrapolation scale of $A_p\in \mathcal{H}\big(W_{p,\mathcal{B}}^2(\Omega),L_p(\Omega)\big)$, see the proof of Theorem~\ref{TTT} and \cite{LQPP,Amann_Teubner} for interpolation-extrapolation scales.\\

Another example are fourth-order operators.  Given $\delta\in\{0,1\}$ define the boundary operator
\begin{subequations}\label{AA}
\begin{equation}
\mathsf{B}u:=(1-\delta)\partial_\nu u+\delta \Delta u
\end{equation}
(i.e. for $\delta=0$ and $\delta=1$ we consider clamped respectively pinned boundary conditions) and set
\begin{equation}
 W_{p,\mathsf{B}}^4(\Omega):=\big\{u\in W_{p}^4(\Omega)\,;\, \mathsf{B}u=0\ \text{on}\ \partial\Omega\big\}
\end{equation}
for $p\in (1,\infty)$. Let $d_1>0$, $d_2\ge 0$ and set
\begin{equation}\label{AAp}
\mathsf{A}_pu:=-d_1\Delta^2u+d_2\Delta u\,,\quad u\in W_{p,\mathsf{B}}^4(\Omega)\,.
\end{equation}
\end{subequations}

Verbatim the same proof as for Corollary~\ref{C1} yields:

\begin{cor}\label{C2}
Assume~\eqref{AA}. For the operator $\mathsf{A}_p\in \mathcal{H}\big(W_{p,\mathsf{B}}^4(\Omega),L_p(\Omega)\big)$ defined in~\eqref{AAp} with $p\in (1,\infty)$ we have:\vspace{2mm}

{\bf (a)}   Condition~\eqref{spect} is satisfied if~ $b\in C([0,T],\R^+)$ with $\displaystyle\int_0^T b(t)\,\rd t> -a$ and $p\in [2,\infty)$.\vspace{2mm}

{\bf (b)} Condition~\eqref{spectE100x} is satisfied if $b\le 1$.\vspace{2mm}

{\bf (c)} Condition~\eqref{spectE200} is satisfied if $b\in C([0,T],\R^+)$.
 
\end{cor}

%%%%%%%%%%%%%%%%%%%%%%%%%%%%%%%%%%%%%%%%%%%%%%%%%%%%%%%% 
%%%%%%%%%%%%%%%%%%%%%%%%%%%%%%%%%%%%%%%%%%%%%%%%%%%%%%%%
\section{Examples for the Nonlinearity $f$}\label{Sec7}
%%%%%%%%%%%%%%%%%%%%%%%%%%%%%%%%%%%%%%%%%%%%%%%%%%%%%%%%
%%%%%%%%%%%%%%%%%%%%%%%%%%%%%%%%%%%%%%%%%%%%%%%%%%%%%%%%

We provide examples for nonlinearities $f$ satisfying~\eqref{A2}.  We first state a slightly more general case of  functions considered in Corollary~\ref{C11}.

\begin{prop}\label{F1}
Let $\ell>0$, $g\in C\big([0,T]\times E_\theta,E_\gamma)$ with $g(t,0)=0$, $t\in [0,T]$, such that
  there is $c>0$ with  
\begin{equation}\label{F2X}
\|g(t,v)-g(t,w)\|_{\gamma}\le c\,\left(\|v\|_\theta^\ell+\|w\|_\theta^\ell\right)\, \|v-w\|_{\theta}\,,\qquad t\in [0,T]\,,\quad v,w\in E_\theta\,.
\end{equation}
Set
$$
f(u)(t):=g\big(t,u(t)\big)\,,\qquad t\in [0,T]\,,\quad u\in C_\theta\big((0,T],E_\theta\big)\,.
$$
Then 
$$
f:C_{\theta}\big((0,T],E_\theta\big)\to C_{\theta(\ell+1)}\big((0,T],E_\gamma\big)
$$ 
 satisfies~\eqref{A2} with $\nu=\theta(\ell+1)$.
\end{prop}

\begin{proof}
This readily follows from the facts that 
$$
t^\theta\|u(t)\|_\theta\le \|u\|_{C_{\theta}((0,t],E_\theta)}\,,\quad t\in (0,T]\,,\qquad \lim_{t\to 0^+}\|u\|_{C_{\theta}((0,t],E_\theta)}= 0\,.
$$
\end{proof}

 The following result on Nemytskii operators is useful for checking condition~\eqref{F2X}:

\begin{lem}\label{g-Lemma}
Let $\Omega$ be an open subset of $\R^n$. Consider $g\in C^1(\R,\R)$ with $g(0)=g'(0)=0$ and
\begin{equation}\label{keyx}
\vert  g'(r)- g'(s)\vert \le c\big(\vert r\vert^{\ell-1}+\vert s\vert^{\ell-1}\big) \vert r-s\vert  \,,\quad r,s\in\R\,,
\end{equation}
for some constants $\ell\ge 1$ and $c>0$. 
Let $p\in [1,\infty)$ and $\mu\in (0,1)$.  Then $g(w)\in W_{p}^{\mu}(\Omega)$ for every~${w\in W_{p}^{\mu}(\Omega)\cap L_\infty (\Omega)}$ and $g(0)=0$.
 Moreover,
there is $K>0$ with
\begin{equation*}
    \begin{split}
\|g(w_1)-g(w_2)\|_{W_{p}^\mu}
     &\le  K \big(\| w_1\|_{\infty}^{\ell}+\| w_2\|_{\infty}^{\ell}\big) \| w_1-w_2\|_{W_{p}^\mu}\\
&\quad
+ K \big(\| w_1\|_{\infty}^{\ell-1}+\| w_2\|_{\infty}^{\ell-1}\big) \big(\| w_1\|_{W_{p}^\mu}+\| w_2\|_{W_{p}^\mu}\big) \| w_1-w_2\|_{\infty}
    \end{split}
    \end{equation*}
for  all $w_1,\, w_2\in W_{p}^{\mu}(\Omega)\cap L_\infty (\Omega)$.
\end{lem}

\begin{proof}
Since $\vert g'(r)\vert\le c\vert r\vert^\ell$ for $r\in \R$, the proof is the same as \cite[Lemma~4.1]{MW24}.
\end{proof}

\subsection*{A Nonlocal Example}

As pointed out before our theory covers the case of nolinearities $f=f(u)$ depending nonlocally with respect to time on $u$. To give an example, consider a kernel
\begin{subequations}\label{h}
\begin{equation}
k\in C\big((0,T]^2\times E_\theta,E_\gamma\big)\,,\qquad k(t,s,0)=0\,,\quad (t,s)\in (0,T]^2\,,
\end{equation}
such that
\begin{equation}
\|k(t,s,v)-k(t,s,w)\|_\gamma\le h_T(t,s)\,\big(\|v\|_\theta^\ell+\|w\|_\theta^\ell\big)\, \|v-w\|_\theta\,,\qquad (t,s)\in (0,T]^2\,,\quad v,w\in E_\theta\,,
\end{equation}
for some $\ell>0$ and some non-negative measurable function $h_T$ on $(0,T]^2$ such that there is $\nu\in [0,1)$ with
\begin{equation}\label{73c}
\left[t\mapsto\int_0^T h_T(t,s)\, s^{-\theta(\ell+1)}\,\rd s\right]\in C_\nu\big((0,T],\R\big)\,.
\end{equation}
\end{subequations}

\begin{lem}
Assume~\eqref{h}. Then, the function $f$, defined as
$$
f(v)(t):=\int_0^T k\big(t,s,v(s)\big)\,\rd s\,,\quad t\in (0,T]\,,\quad v\in C_\theta\big((0,T],E_\theta\big)\,,
$$
satisfies~\eqref{A2}.
\end{lem}

\begin{proof}
The assertion follows from the observation
\begin{align*}
\|f(v)(t)-&f(w)(t)\|_\gamma\le \int_0^T h_T(t,s)\,\big(\|v(s)\|_\theta^\ell+\|w(s)\|_\theta^\ell\big)\, \|v(s)-w(s)\|_\theta\,\rd s\\
&\le \int_0^T h_T(t,s)\,s^{-\theta(\ell+1)}\,\rd s\, \big(\|v\|_{C_\theta((0,T]),E_\theta)}^\ell+\|w\|_{C_\theta((0,T]),E_\theta)}^\ell\big)\, \|v-w\|_{C_\theta((0,T]),E_\theta)}
\end{align*}
for $t\in (0,T]$ and $v,w\in C_\theta((0,T]),E_\theta)$.
\end{proof}

For instance, $h_T(t,s)= c(t-s)_+^\lambda$ with $c>0$ satisfies~\eqref{73c} provided that $\lambda>-1$, $\theta(\ell+1)<1$, and $1+\nu+\lambda > \theta(\ell+1)$.\\

%%%%%%%%%%%%%%%%%%%%%%%%%%%%%%%%%%%%%%%%%%%%%%%%%%%%%%%%%%%%%%%%%
%%%%%%%%%%%%%%%%%%%%%%%%%%%%%%%%%%%%%%%%%%%%%%%%%%%%%%%%%%%%%%%%%
\section{Revisiting the  Diffusion Equation~\eqref{DD}}\label{Sec8}
%%%%%%%%%%%%%%%%%%%%%%%%%%%%%%%%%%%%%%%%%%%%%%%%%%%%%%%%%%%%%%%%%
%%%%%%%%%%%%%%%%%%%%%%%%%%%%%%%%%%%%%%%%%%%%%%%%%%%%%%%%%%%%%%%%%

Let $\Omega\subset\R^n$ be an open and bounded (sufficiently smooth) domain. We consider the semilinear version of problem~\eqref{DD}
\begin{subequations}\label{exx}
\begin{align}
\partial_t u-\mathrm{div}\big(d(x)\nabla u\big)+c(x) u &= g(t,x,u)\,,\qquad (t,x)\in (0,T]\times\Omega\,,\label{81a}\\
(1-\delta)u+\delta \partial_\nu u&=0\,,\qquad  (t,x)\in (0,T]\times\partial\Omega\,,\\
a\, u(T)+\int_0^T  b(t)\, u(t)\,\rd t&=M\,,
\end{align}
\end{subequations}
where $\delta\in\{0,1\}$ is fixed and where we impose that
\begin{subequations}\label{Assu}
\begin{equation}\label{key}
\begin{split}
&d\in C^1\big(\bar\Omega,\R_{sym}^{n\times n}\big) \ \text{ with }\  d(x)\zeta\cdot\zeta\ge \underline{d}\, \vert\zeta\vert^2\,,\quad (\zeta,x)\in\R^n\times \bar\Omega\,,\\
&c\in C(\bar\Omega,\R^+)\ \text{ with \ $\min_{\bar\Omega} c>0$ if $\delta=1$}\,,\\
&\text{$a\in \R$,  $b\in C^1([0,T],\R^+)$ with  $b(0)>0$ and $\int_0^T b(t)\,\rd t> -a$}\,. 
\end{split}
\end{equation}
Moreover, we assume that 
\begin{align}
g\in  C^{2}\big([0,T]\times\bar\Omega\times\R\big)\,,\qquad g(t,x,0)=0\,,\quad (t,x)\in [0,T]\times\bar\Omega\,,
\end{align}
and that there is $\ell>0$ with
\begin{align}\label{ggg}
\vert g(t,x,v)-g(t,x,w)\vert\le c\,\big(\vert v\vert^\ell+\vert w\vert^\ell\big)\,\vert v-w\vert\,,\quad \quad (t,x)\in [0,T]\times\bar\Omega\,,\quad v,w\in\R\,.
\end{align}
\end{subequations}
Then we can prove the following well-posedness result for problem~\eqref{exx}:

\begin{thm}\label{TTT}
Assume~\eqref{Assu} and let $p>n$ with $p\ge 2$ and 
\begin{equation}\label{81}
\max\left\{ \frac{n}{p}-\frac{2}{\ell+1}\,,\,-\frac{3}{2}+\delta\right\}<2\alpha< 0\,,\qquad 2\alpha\not=-1-\delta+\frac{1}{p}\,.
\end{equation}
There is $m(T)>0$ such that for each 
$$
\text{$M\in H_{p,\mathcal{B}}^{2\alpha+2}(\Omega)$\ with\   $\|M\|_{H_{p}^{2\alpha+2}}\le m(T)$},
$$ 
problem~\eqref{exx} admits a unique strong solution
$$
u\in C^1\big((0,T],L_p(\Omega)\big)\cap C\big((0,T],W_{p,\mathcal{B}}^2(\Omega)\big)\cap C\big([0,T],L_p(\Omega)\big)\,.
$$
%with $2\theta\in \big(\frac{n}{p}-\sigma+2,\frac{2}{q+1}\big)$.
\end{thm}

\begin{proof}
Let us first observe that~\eqref{ggg} implies
\begin{align}\label{estim}
\|g(t,v)-g(t,w)\|_{L_p}\le c\,\big(\|v\|_\infty^\ell+\|w\|_\infty^\ell\big)\,\|v-w\|_{L_p}\,,\quad v,w\in L_\infty(\Omega)\,,
\end{align}
where we use the shorthand $g(t,v)(x):=g(t,x,v(x))$.
Based on this estimate, we now aim at writing problem~\eqref{exx} in a form that fits into the framework of Theorem~\ref{T0}. This form will require, as a byproduct, only low regularity assumptions on $M$.

Let $q\in \{2,p\}$. We start by introducing as in \eqref{Ap} the operator $A_q\in\mathcal{H}\big(W_{q,\mathcal{B}}^2(\Omega),L_q(\Omega)\big)$ by 
\begin{equation*}
A_qu:=\mathrm{div}\big(d(x)\nabla u\big)-c(x) u\,,\quad u\in W_{q,\mathcal{B}}^2(\Omega)\,,
\end{equation*}
and note that $A_2=(A_2)^*\le \lambda_1-\underline{c}<0$, where $\lambda_1$  is the first eigenvalue of $A_2+c$ (with $\lambda_1<0$ if $\delta=0$ and $\lambda_1=0$ if $\delta=1$) and $\underline{c}:=\min_{\bar\Omega} c$. Also note from~\cite{AmannIsrael} that $\sigma(A_2)=\sigma(A_p)$. In fact, using the interpolation-extrapolation scale of order 1 corresponding to $A_q\in\mathcal{H}\big(W_{q,\mathcal{B}}^2(\Omega),L_q(\Omega)\big)$ with the complex interpolation functor (see \cite[Theorem~7.1; Equation (7.5)]{Amann_Teubner} and \cite[\S V.1]{LQPP}), which  gives the scale
\begin{equation*}
H_{q,\mathcal{B}}^{2\theta}(\Omega):=\left\{\begin{array}{ll} \{v\in H_{q}^{2\theta}(\Omega) \,:\, \mathcal{B} v=0 \text{ on } 
 \partial\Omega\}\,, &\delta+\frac{1}{q}<2\theta\le 2 \,,\\[3pt]
	 H_{q}^{2\theta}(\Omega)\,, & -2+\frac{1}{q}+\delta< 2\theta<\delta +\frac{1}{q}\,,\end{array} \right.
\end{equation*}
 and denoting by
$$
A_{q,\alpha}\in\mathcal{H}\big(H_{q,\mathcal{B}}^{2\alpha+2}(\Omega),H_{q,\mathcal{B}}^{2\alpha}(\Omega)\big)
$$
the $H_{q,\mathcal{B}}^{2\alpha}(\Omega)$-realization of $A_q$ (note that $-2+\delta+1/q\le -3/2+\delta<2\alpha$), it follows from~\cite[V.Theorem~2.1.3]{LQPP} that 
$$
\sigma(A_{q,\alpha})=\sigma(A_q)=\sigma(A_2)\,.
$$
Notice that $H_{q,\mathcal{B}}^{2\alpha+2}(\Omega)$ embeds compactly into $H_{q,\mathcal{B}}^{2\alpha}(\Omega)$ and that $A_{2,\alpha}=(A_{2,\alpha})^*\le\lambda_1-\underline{c}<0$ due to~\cite[V.Theorem~1.5.15]{LQPP}; that is, $A_{2,\alpha}$ satisfies~\eqref{As} (with $E_0=H_{2,\mathcal{B}}^{2\alpha}(\Omega))$ and $\alpha<0$ there). Therefore, 
setting
\begin{equation*}
\Phi_T^{(q,\alpha)}:=\left(a\,e^{TA_{q,\alpha}}+\int_0^Tb(t)\,e^{tA_{q,\alpha}}\,\rd t\right)\in \mathcal{L}\big(H_{q,\mathcal{B}}^{2\alpha}(\Omega),H_{q,\mathcal{B}}^{2\alpha+2}(\Omega)\big) \,,
\end{equation*}
we infer from Proposition~\ref{PP1}~{\bf (a)} and \eqref{key} that $\mathrm{ker}\big(\Phi_T^{(2,\alpha)}\big)=\{0\}$. Since $e^{TA_{2,\alpha}}\big\vert_{H_{p,\mathcal{B}}^{2\alpha}(\Omega)}=e^{TA_{p,\alpha}}$ as $p\ge 2$, we have $\Phi_T^{(p,\alpha)}\subset \Phi_T^{(2,\alpha)}$ and thus  
\begin{equation}\label{t2a}
\mathrm{ker}\big(\Phi_T^{(p,\alpha)}\big)=\{0\}\,.
\end{equation} 
That is,~$A_{p,\alpha}$ satisfies~\eqref{spect}.
Set $E_j:=H_{p,\mathcal{B}}^{2\alpha+2j}(\Omega)$ for $j=0,1$ so that 
\begin{equation}\label{t1}
A_{p,\alpha}\in\mathcal{H}(E_1,E_0)
\end{equation} 
and put
$$
E_\theta:=[E_0,E_1]_\theta=\big[H_{p,\mathcal{B}}^{2\alpha}(\Omega),H_{p,\mathcal{B}}^{2\alpha+2}(\Omega)\big]_\theta\,.
$$
It then follows from \cite[Theorem~2.3]{Guidetti91} and the chain of embeddings
$$H_p^{s+\ve}(\Omega)\hookrightarrow B_{p,p}^{s}(\Omega)\hookrightarrow H_p^{s-\ve}(\Omega)\,,\quad s\in\R\,,\quad \ve>0\,,
$$
that
\begin{equation}\label{interpol}
H_{p,\mathcal{B}}^{2\alpha+2\theta+\ve}(\Omega)\hookrightarrow E_\theta\hookrightarrow  H_{p,\mathcal{B}}^{2\alpha+2\theta-\ve}(\Omega)\,,\quad 2\theta\in [0,2]\,,\quad 0<\ve\ll 1\,.
\end{equation}
 Owing to~\eqref{81} we may choose $2\theta,2\gamma\in (0,2)$ and $\ve\in (0,1)$ such that
$$
n/p-2\alpha+\ve \le 2\theta<2/(\ell+1)\,,\qquad   2\alpha+2\gamma+\ve<0\,.
$$ 
Then \eqref{interpol} yields
$$
L_p(\Omega)\hookrightarrow H_{p,\mathcal{B}}^{2\alpha+2\gamma+\ve}(\Omega)\hookrightarrow E_{\gamma}\,,\qquad
E_\theta\hookrightarrow  H_{p,\mathcal{B}}^{2\alpha+2\theta+\ve}(\Omega)\hookrightarrow L_\infty(\Omega)\,,
$$
so that ~\eqref{estim} entails
\begin{align*}
\|g(t,v)-g(t,w)\|_{E_\gamma}\le c\,\big(\|v\|_{E_\theta}^\ell+\|w\|_{E_\theta}^\ell\big)\,\|v-w\|_{E_\theta}\,,\quad v,w\in {E_\theta}\,.
\end{align*}
Now, Proposition~\ref{F1} implies that
\begin{align}\label{t2}
\|g(\cdot,v)-g(\cdot,w)\|_{C_{\nu}((0,T],E_\gamma)}\le c\,\big(\|v\|_{C_\theta((0,T],E_\theta)}^{\ell}+\|w\|_{C_\theta ((0,T],E_\theta)}^{\ell}\big)\,\|v-w\|_{C_\theta((0,T],E_\theta)}
\end{align}  
for $ v,w\in C_\theta((0,T],E_\theta)$ with $\nu:=\theta(\ell+1)\in (0,1)$. That is, $g$ satisfies~\eqref{A2}.
Consequently, considering~\eqref{exx} in the form
$$
u'=A_{p,\alpha}u+g(t,u)\,,\quad t\in (0,T]\,,\qquad a u(T)+\int_0^Tb(t)\,u(t)\,\rd t=M\,,
$$
in $E_0=H_{p,\mathcal{B}}^{2\alpha}(\Omega)$, it follows from~\eqref{t2a}~-~\eqref{t2} that we may apply Theorem~\ref{T0} to deduce that there is $m(T)>0$ such that for each 
$M\in H_{p,\mathcal{B}}^{2\alpha+2}(\Omega)$ with   $\|M\|_{H_{p}^{2\alpha+2}}\le m(T)$,
 problem~\eqref{exx} admits a unique strong solution
$$
u\in C^1\big((0,T],H_{p,\mathcal{B}}^{2\alpha}(\Omega)\big)\cap C\big((0,T],H_{p,\mathcal{B}}^{2\alpha+2}(\Omega)\big)\cap C\big([0,T],H_{p,\mathcal{B}}^{2\alpha}(\Omega)\big)\cap C_\theta\big((0,T],H_{p,\mathcal{B}}^{2\alpha+2\theta}(\Omega)\big)\,.
$$
To improve its regularity note that $u_\ve:=u(\ve+\cdot)\in C\big([0,T-\ve],H_{p,\mathcal{B}}^{2\alpha+2}(\Omega)\big)$ for $\ve \in (0,T-\ve)$ is a mild solution to the linear Cauchy problem
$$
u'=A_{p}u+g_\ve(t)\,,\quad t\in (0,T]\,,\qquad u_\ve(0)=u(\ve)\,,
$$
with $g_\ve(t):=g(\ve+t,u_\ve(t))$. 
Since $2\alpha+2>n/p$ and since $g\in C^{2}$, it follows e.g. from~\cite[Lemma~2.7]{WalkerEJAM} that $g_\ve\in  C\big([0,T-\ve],H_{p,\mathcal{B}}^{\mu}(\Omega)\big)$ for some $\mu \in (n/p,1)$. Hence,~\cite[II.Theo\-rem~1.2.2]{LQPP} yields
$$
u_\ve\in C^1\big((0,T-\ve],L_p(\Omega)\big)\cap C\big((0,T-\ve],W_{p,\mathcal{B}}^2(\Omega)\big)\cap C\big([0,T-\ve],L_p(\Omega)\big)\,,
$$
and since $\ve\in (0,T)$ was arbitrary, the assertion follows.
\end{proof}

Note that we use in the proof of Theorem~\ref{TTT} the interpolation-extrapolation scale in order to cope with the requirement of Theorem~\ref{T0} that $\gamma>0$. The latter enforces to shift the  operator $A_p\in \mathcal{H}\big(W_{p,\mathcal{B}}^2(\Omega),L_p(\Omega)\big)$ to negative spaces, i.e. to use  $A_{p,\alpha}\in\mathcal{H}\big(H_{p,\mathcal{B}}^{2\alpha+2}(\Omega),H_{p,\mathcal{B}}^{2\alpha}(\Omega)\big)$ with $\alpha<0$ instead.\\ 

In the Dirichlet case $\delta=0$ and if $\ell\in (0,1)$, one may take $2\alpha=-1$ for~$p$ sufficiently large (i.e. $p$ larger than $n(1+\ell)/(1-\ell)$) and then obtains a smallness condition on $M\in W_{p,D}^1(\Omega)=H_{p,D}^1(\Omega)$.\\

It is also worth pointing out that one may consider equations with gradient-dependent nonlinearities of the form
$$
\partial_t u-\mathrm{div}\big(d(x)\nabla u\big)+c(x) u = g(t,x,u,\nabla u)\,,\qquad t\in (0,T]\,,\quad x\in\Omega\,,
$$
instead of \eqref{81a}. One derives the same result as stated in Theorem~\ref{TTT} provided one replaces~\eqref{ggg} by
$$
\vert g(t,x,v,\eta)-g(t,x,w,\xi)\vert\le c\,\big(\vert v\vert^\ell+\vert w\vert^\ell+\vert \eta\vert^\ell+\vert \xi\vert^\ell\big)\,\big(\vert v-w\vert+\vert \eta-\xi\vert\big)\,, 
$$
for the function $g\in  C^{2}\big([0,T]\times\bar\Omega\times\R\times\R^n\big)$ with $\ell \in (0,1)$ and $g(t,x,0,0)=0$ and provided one replaces~\eqref{81} by
$$
\max\left\{ 1+\frac{n}{p}-\frac{2}{\ell+1}\,,\,-\frac{3}{2}+\delta\right\}<2\alpha< 0\,,\qquad 2\alpha\not=-1-\delta+\frac{1}{p}\,.
$$
The proof is the same.

\bibliographystyle{siam}
\bibliography{Literature}

\begin{thebibliography}{10}

\bibitem{AmannIsrael}
{\sc H.~Amann}, {\em Dual semigroups and second order linear elliptic boundary
  value problems}, Israel J. Math., 45 (1983), pp.~225--254.

\bibitem{Amann_Teubner}
\leavevmode\vrule height 2pt depth -1.6pt width 23pt, {\em {Nonhomogeneous
  linear and quasilinear elliptic and parabolic boundary value problems}}, in
  {Function spaces, differential operators and nonlinear analysis
  ({F}riedrichroda, 1992)}, vol.~133 of {Teubner-Texte Math.}, Teubner,
  Stuttgart, 1993, p.~9–126.

\bibitem{LQPP}
\leavevmode\vrule height 2pt depth -1.6pt width 23pt, {\em Linear and
  quasilinear parabolic problems. {V}ol. {I}}, vol.~89 of Monographs in
  Mathematics, Birkh\"{a}user Boston, Inc., Boston, MA, 1995.
\newblock Abstract linear theory.

\bibitem{AuEtal_JIEA_20}
{\sc V.~V. Au, Y.~Zhou, N.~H. Can, and N.~H. Tuan}, {\em Regularization of a
  terminal value nonlinear diffusion equation with conformable time
  derivative}, J. Integral Equations Appl., 32 (2020), pp.~397--416.

\bibitem{Dokuchaev4}
{\sc N.~Dokuchaev}, {\em Estimates for distances between first exit times via
  parabolic equations in unbounded cylinders}, Probab. Theory Related Fields,
  129 (2004), pp.~290--314.

\bibitem{Dokuchaev6}
\leavevmode\vrule height 2pt depth -1.6pt width 23pt, {\em Parabolic {I}to
  equations with mixed in time conditions}, Stoch. Anal. Appl., 26 (2008),
  pp.~562--576.

\bibitem{Dokuchaev7}
\leavevmode\vrule height 2pt depth -1.6pt width 23pt, {\em On prescribed change
  of profile for solutions of parabolic equations}, J. Phys. A, 44 (2011),
  pp.~225204, 7.

\bibitem{Dokuchaev8}
\leavevmode\vrule height 2pt depth -1.6pt width 23pt, {\em On forward and
  backward {SPDE}s with non-local boundary conditions}, Discrete Contin. Dyn.
  Syst., 35 (2015), pp.~5335--5351.

\bibitem{Dokuchaev}
\leavevmode\vrule height 2pt depth -1.6pt width 23pt, {\em On recovering
  parabolic diffusions from their time-averages}, Calc. Var. Partial
  Differential Equations, 58 (2019), pp.~Paper No. 27, 14.

\bibitem{EngelNagel}
{\sc K.-J. Engel and R.~Nagel}, {\em A short course on operator semigroups},
  Universitext, Springer, New York, 2006.

\bibitem{Guidetti91}
{\sc D.~Guidetti}, {\em On interpolation with boundary conditions}, Math. Z.,
  207 (1991), pp.~439--460.

\bibitem{MW24}
{\sc B.-V. Matioc and {\relax Ch}.~Walker}, {\em Well-posedness of quasilinear
  parabolic equations in time-weighted spaces}.
\newblock arXiv:2312.07974, 2023.

\bibitem{PSW18}
{\sc J.~Prüss, G.~Simonett, and M.~Wilke}, {\em Critical spaces for
  quasilinear parabolic evolution equations and applications}, J. Differential
  Equations, 264 (2018), p.~2028–2074.

\bibitem{TachEtal_MMAS23}
{\sc T.~N. Thach, N.~H. Can, and V.~V. Tri}, {\em Identifying the initial state
  for a parabolic diffusion from their time averages with fractional
  derivative}, Math. Methods Appl. Sci., 46 (2023), pp.~7751--7766.

\bibitem{TranEtal_MMAS20}
{\sc N.~Tran, V.~V. Au, Y.~Zhou, and N.~H. Tuan}, {\em On a final value problem
  for fractional reaction-diffusion equation with {R}iemann-{L}iouville
  fractional derivative}, Math. Methods Appl. Sci., 43 (2020), pp.~3086--3098.

\bibitem{TrietEtal_MMAS20}
{\sc N.~A. Triet, V.~V. Au, L.~D. Long, D.~Baleanu, and N.~H. Tuan}, {\em
  Regularization of a terminal value problem for time fractional diffusion
  equation}, Math. Methods Appl. Sci., 43 (2020), pp.~3850--3878.

\bibitem{TrietEtal_MMAS21}
{\sc N.~A. Triet, T.~T. Binh, N.~D. Phuong, D.~Baleanu, and N.~H. Can}, {\em
  Recovering the initial value for a system of nonlocal diffusion equations
  with random noise on the measurements}, Math. Methods Appl. Sci., 44 (2021),
  pp.~5188--5209.

\bibitem{TuanEtAl_MMAS20}
{\sc N.~H. Tuan, L.~N. Huynh, D.~Baleanu, and N.~H. Can}, {\em On a terminal
  value problem for a generalization of the fractional diffusion equation with
  hyper-{B}essel operator}, Math. Methods Appl. Sci., 43 (2020),
  pp.~2858--2882.

\bibitem{WalkerEJAM}
{\sc {\relax Ch}.~Walker}, {\em Global existence for an age and spatially
  structured haptotaxis model with nonlinear age-boundary conditions}, European
  J. Appl. Math., 19 (2008), pp.~113--147.

\bibitem{Yagi10}
{\sc A.~Yagi}, {\em Abstract parabolic evolution equations and their
  applications}, Springer Monographs in Mathematics, Springer-Verlag, Berlin,
  2010.

\end{thebibliography}
\end{document}